\documentclass[a4paper,12pt]{amsart}

\usepackage[english]{babel}
\usepackage{amssymb,amsfonts,amsxtra,amsmath}
\usepackage{dsfont,mathrsfs}
\renewcommand{\mathcal}{\mathscr}
\usepackage[all]{xypic}
\xyoption{dvips}
\usepackage{url}
\usepackage{csquotes}
\usepackage{fullpage}
\usepackage[colorlinks]{hyperref}
\usepackage[neveradjust]{paralist}
\usepackage{mathtools}

\theoremstyle{definition}
\newtheorem{ntn}{Notation}[section]
\newtheorem{dfn}[ntn]{Definition}
\theoremstyle{plain}
\newtheorem{lem}[ntn]{Lemma}
\newtheorem{prp}[ntn]{Proposition}
\newtheorem{thm}[ntn]{Theorem}
\newtheorem{cor}[ntn]{Corollary}

\theoremstyle{remark}

\newtheorem{rmk}[ntn]{Remark}

\numberwithin{equation}{section}

\newcommand{\ideal}[1]{{\left\langle#1\right\rangle}}
\newcommand{\xymat}{\SelectTips{cm}{}\xymatrix}
\newcommand{\into}{\hookrightarrow}
\newcommand{\onto}{\twoheadrightarrow}

\newcommand{\ol}{\overline}

\newcommand{\wh}{\widehat}
\newcommand{\wt}{\widetilde}

\newcommand{\A}{\mathcal{A}}
\newcommand{\C}{\mathcal{C}}

\newcommand{\F}{\mathcal{F}}
\newcommand{\I}{\mathcal{I}}
\newcommand{\J}{\mathcal{J}}
\renewcommand{\L}{\mathcal{L}}
\newcommand{\M}{\mathcal{M}}
\newcommand{\mm}{\mathfrak{m}}
\newcommand{\nn}{\mathfrak{n}}
\newcommand{\pp}{\mathfrak{p}}
\newcommand{\qq}{\mathfrak{q}}
\newcommand{\reg}{\mathrm{reg}}
\renewcommand{\O}{\mathcal{O}}
\newcommand{\Q}{\mathcal{Q}}

\DeclareMathOperator{\Ann}{Ann}
\DeclareMathOperator{\Ass}{Ass}
\DeclareMathOperator{\Bl}{Bl}

\DeclareMathOperator{\depth}{depth}
\DeclareMathOperator{\grade}{grade}
\DeclareMathOperator{\Hom}{Hom}
\DeclareMathOperator{\SHom}{{\mathcal Hom}}
\DeclareMathOperator{\Proj}{Proj}
\DeclareMathOperator{\rk}{rk}
\DeclareMathOperator{\Spec}{Spec}

\begin{document}

\title{Blow up of conductors}

\author[C.~Birghila]{Corina Birghila}
\address{C.~Birghila\\
Department of Statistics and Operations Research\\ 
University of Vienna\\
Oskar-Morgenstern-Platz 1\\
1090 Vienna\\
Austria}
\email{\href{mailto:corina.birghila@univie.ac.at}{corina.birghila@univie.ac.at}}


\author[M.~Schulze]{Mathias Schulze}
\address{M.~Schulze\\
Department of Mathematics\\
TU Kaiserslautern\\
67663 Kaiserslautern\\
Germany}
\email{\href{mailto:mschulze@mathematik.uni-kl.de}{mschulze@mathematik.uni-kl.de}}

\thanks{The research leading to these results has received funding from the People Programme (Marie Curie Actions) of the European Union's Seventh Framework Programme (FP7/2007-2013) under REA grant agreement n\textsuperscript{o} PCIG12-GA-2012-334355.}




\begin{abstract}
We generalize results of P.M.H.~Wilson describing situations where the blow up of the conductor ideal of a scheme coincides with the normalization.
\end{abstract}

\maketitle

\section{Introduction}\label{53}

Blowup and normalization are fundamental operations in the study of varieties and singularities.
While normalization modifies the non-normal locus defined by the conductor ideal, blow up modifies the locus defined by any given ideal.
In typical cases the normalization is finite while blow ups are not.
It is therefore a particular situation that the blow up of the conductor ideal yields the normalization.
P.M.H.~Wilson described instances where this happens.
He considers irreducible projective varieties over an algebraically closed field and proves the following results (see \cite[Cor.~1.4, Thm.~2.7, Rem.~2.8]{Wil78}).

\begin{prp}[Wilson]
Given a curve $C$ with normalization $\wt C$ and with $C'$ the blow up of $C$ in its conductor ideal, then $C'=\wt C$.
\end{prp}

\begin{thm}[Wilson]
The blow up $V'$ of a hypersurface in its conductor ideal $\C$ is the same as the normalization $\wt V$ if and only if the dualizing sheaf $\omega_{\wt V}$ is invertible.
In particular, if $V$ is a surface, then $V'=\wt V$ if and only if $\wt V$ is Gorenstein.\qed
\end{thm}

Ragni Piene generalized the \enquote{if}-part of Wilson's theorem to reduced (not necessarily irreducible) algebraic schemes over an algebraically closed field replacing the normalization by a finite birational morphism (see \cite[Prop.~(2.9)]{Pie78}).

\begin{prp}[Piene]\label{3}
Let $f\colon Y\to X$ be a finite, birational morphism between Gorenstein
schemes. 
Then $f$ is isomorphic to the blow up of the conductor of $Y$ in $X$.
\end{prp}

In this note we further generalize Wilson's results dropping the base field.
We consider Cohen--Macaulay schemes equipped with a canonical (fractional) ideal sheaf.
Our main result Theorem~\ref{47} generalizes Piene's result and yields

\begin{thm}\label{48}
Let $X$ be a reduced Gorenstein Nagata scheme with Cohen--Macaulay normalization $\wt X\to X$.
Denote by $\Bl_{\C_{\wt X/X}} X$ the blow up of $X$ in the conductor ideal $\C_{\wt X/X}$.
Then $\wt X=\Bl_{\C_{\wt X/X}} X$ if and only if $\wt X$ is Gorenstein.\qed
\end{thm}

The above mentioned main result involves the blow up of fractional ideals. 
In preparation, we collect results on sheaves of rational functions and consider morphisms that allow for a pullback of fractional ideals.
Up to some extent we describe these concepts in relation with associated points of schemes.
A slightly different account of this topic is given in \cite[\S7.1]{Liu02}.
Although we work with sheaves on locally Noetherian schemes, our results are mostly local in the realm of commutative algebra.

The question under consideration also appears in work of Mitsuo Shinagawa \cite{Shi82} that aims for deducing properties of a scheme from its normalization.
Under the strong condition of normal flatness (that we do not pursue here) he proves 

\begin{thm}[M.~Shinagawa]
Let $X$ be a reduced Noetherian scheme with finite normalization $\wt X$, $Y$ be the closed subscheme defined by the conductor of $X$ in $\wt X$, and $X'$ the blow up of $X$ in $Y$.
If $X$ is normally flat along $Y$ and $Y$ is of pure codimension $1$ in $X$, then $X'$ is naturally isomorphic to $X$.\qed
\end{thm}

\subsection*{Acknowledgments}
Preliminary results towards the ones presented here were obtained in the first named author's Master's thesis~\cite{Bir14}.

\section{Rational functions}\label{49}

All rings will be Noetherian commutative rings with unity.
For a ring $A$ we denote by $A^\reg$ the set of its regular elements and by
\[
Q(A):=(A^\reg)^{-1}A
\]
its total ring of fractions.
All schemes will be locally Noetherian and all morphisms quasicompact, that is, locally on the target, morphisms of Noetherian schemes.
A property that holds over each affine open set is refered to as an affine local property.

Let $X$ be a scheme.
Then $x\in X$ is called an \emph{associated point of $X$} if $\mm_{X,x}$ is an associated prime of $\O_{X,x}$.
We denote by $\Ass X$ the (locally finite) set of associated points of $X$.
For $x\in X$ we set
\[
\Ass(X,x):=\Ass(\O_{X,x}).
\]
Note that $U\cap\Ass X=\Ass U$ for any open $U\subset X$.
The following result is well-known; we give a proof.

\begin{lem}\label{7}
If $X=\Spec A$ is affine, then $\Ass X=\Ass A$.
\end{lem}

\begin{proof}\
\begin{asparaitem}

\item[($\subset$)] Let $\pp=\ideal{p_1,\dots,p_n}\in\Ass X$.
This means that $\pp A_\pp=\Ann_{A_\pp}(g/1)$ for some $g/1\in A_\pp$.
Then $g\in\pp$ and there are $q_i\not\in\pp$ such that $p_iq_ig=0$ in $A$.
It follows that $\pp\subset\Ann_A(qg)$ where $q=q_1\cdots q_n\not\in\pp$. 
Conversely, let $r\in\Ann_A(qg)$, then $r/1\in\Ann_A(qg)_\pp=\Ann_{A_\pp}(g/1)=\pp A_\pp$ implies $r\in\pp$. 
Thus, $\pp=\Ann_A(qg)$ which means that $\pp\in\Ass A$.

\item[($\supset$)] Let $\pp\in\Ass A$.
Then there is an inclusion $A/\pp\into A$ and hence $A_\pp/\pp A_\pp\into A_\pp$ by exactness of localization.
This means that $\pp\in\Ass X$.\qedhere

\end{asparaitem}
\end{proof}

For $x,y\in X$, we say that $y$ \emph{specializes to} $x$ (or $x$ \emph{generalizes to} $y$) and write $y\leadsto x$ if $x$ is in the closure of $y$.
This makes $X$ and hence $\Ass X$ into poset by setting $x\ge y$ if and only if $y\leadsto x$.
For $X=\Spec(A)$ and $x=\pp$ and $y=\qq$ this is equivalent to $\qq\subset\pp$.

\begin{lem}\label{56}
Any point of a locally Noetherian scheme specializes to a closed point.
\end{lem}

\begin{proof}
See \cite[\href{http://stacks.math.columbia.edu/tag/02IL}{Lem.~02IL}]{stacks-project}.
\end{proof}

We equip $\Ass X\subset X$ with the subspace Zariski topology.
By the following result it consists of all decreasing subsets, that is, subsets stable under generalization.

\begin{lem}\label{41}
For each $x\in X$,
\begin{equation}\label{6}
\Ass(X,x)=\{y\in\Ass X\mid y\leadsto x\}.
\end{equation}
In particular, $\Ass(X,x)\subset\Ass X$ is open and equals the intersection of $\Ass X$ with all open neighborhoods of $x\in X$.
In case $x\in\Ass X$ this means that $\Ass(X,x)$ is the smallest open neighborhood of $x\in\Ass X$.
\end{lem}

\begin{proof}
Replacing $X$ by an affine open neighborhood of $x$ we may assume that $X=\Spec A$ is affine and we write $\pp$ for $x$.
In particular, $\Ass(X,x)=\Ass A_\pp$ and $\Ass X=\Ass A$ by Lemma~\ref{7}.

\begin{asparaitem}

\item[($\subset$)] Let $\qq'\in\Ass A_\pp$ correspond to $\qq\in\Spec A$ with $\qq\subset\pp$.
Then there is an inclusion $A_\pp/\qq A_\pp=A_\pp/\qq'\into A_\pp$ and hence $A_\qq/\qq A_\qq\into A_\qq$ by exactness of localization.
This means that $\qq\in\Ass A$.

\item[($\supset$)] Let $\qq\in\Ass A$ with $\qq\subset \pp$.
This means that there is an inclusion $A/\qq\into A$ and hence $A_\pp/\qq A_\pp\subset A_\pp$ by exactness of localization.
This means that $\qq'=\qq A_\pp\in\Ass A_\pp$.

\end{asparaitem}

Let now $\{\qq\in\Ass A\mid\qq\not\subset\pp\}=\{\qq_1,\dots,\qq_n\}$.
Pick $f_i\in\qq_i\setminus\pp$ and set $f:=f_1\cdots f_n$.
Then $\{\qq\in\Ass A\mid\qq\subset\pp\}=D(f)\cap\Ass A$ is open in $\Ass A$.
\end{proof}

\begin{dfn}\label{27}
The $\O_X$-algebra of \emph{rational functions on $X$} can be defined by
\begin{equation}\label{30}
\Q_X:=i_*i^{-1}\O_X
\end{equation}
where $i:\Ass X\to X$ denotes the inclusion.
\end{dfn}

By \cite{Kle79}, $\Q_X$ is quasicoherent for reduced $X$ (but not in general). 
Moreover, its stalks and sections over affine open sets can be described as follows (see \cite[(20.2.11.1)]{EGA4}).

\begin{lem}\label{4}\pushQED{\qed}
Let $X$ be a scheme. 
\begin{enumerate}[(a)]
\item\label{4b} We have $\Gamma(U,\Q_X)=Q(A)$ for any affine open $U=\Spec A\subset X$.
\item\label{4a} We have $\Q_{X,x}=Q(\O_{X,x})$ for any $x\in X$, hence $\Q_{X,x}=\O_{X,x}$ if $x\in\Ass X$.
\end{enumerate}
\end{lem}

\begin{proof}\pushQED{\qed}
\begin{asparaenum}[(a)]

\item Recall that the $D(t)=\{\pp\in\Spec A\mid t\not\in\pp\}$ for $t\in A$ form a basis of the Zariski topology on $U$.
By Lemma~\ref{7}, 
\begin{equation}\label{39}
A^\reg=\{t\in A\mid D(t)\supset\Ass U\}
\end{equation}
The set $S:=A^\reg$ is multiplicatively closed and directed by setting $t\le t'$ if and only if $t\mid t'$.
For any $t,t'\in S$ with $t\le t'$, there is a morphism $A_t\to A_{t'}$.
These morphisms form a directed system and, using \eqref{39},
\begin{equation}\label{37}
Q(A)=S^{-1}A=\varinjlim_{t\in S}A_t=\varinjlim_{t\in S}\Gamma(D(t),\O_X)=\varinjlim_{D(t)\supset\Ass U}\Gamma(D(t),\O_X).
\end{equation}
By \eqref{30},
\begin{equation}\label{38}
\Gamma(U,\Q_X)=\Gamma(U\cap\Ass X,i^{-1}\O_X)=\varinjlim_{V\supset\Ass U}\Gamma(V,\O_X).
\end{equation}
Combining \eqref{37} and \eqref{38} yields a natural morphism $Q(A)\to\Gamma(U,\Q_X)$.
Conversely, for any open subset $V\supset\Ass U$, prime avoidance yields a $t\in A$ such that $V\supset D(t)\supset\Ass U$.
The claim follows.

\item We may assume that $X=\Spec A$ is affine.
Then, using \eqref{4b}, \eqref{37} and Lemma~\ref{41},
\begin{align*}
\Q_{X,x}&
=\varinjlim_{x\in D(s)}\Gamma(D(s),\Q_X)
=\varinjlim_{x\in D(s)}Q(A_s)\\
&=\varinjlim_{x\in D(s)}\varinjlim_{D(st)\supset\Ass(D(s))}A_{st}
=\varinjlim_{D(s)\supset\Ass(X,x)}A_s
=Q(\O_{X,x}).\qedhere
\end{align*}

\end{asparaenum}
\end{proof}

In particular, it follows from Lemma~\ref{4}.\eqref{4a} that
\[
\O_X\into\Q_X.
\]
We shall describe sections of $\Q_X$, and more generally of $\M\otimes_{\O_X}\Q_X$ for coherent $\M$, over arbitrary open sets.
We abbreviate 
\[
X':=\Ass X,\quad X'_x:=\Ass(X,x).
\]
Lemma~\ref{41} shows that
\begin{equation}\label{28}
(i_*\F)_x=\F(X'_x)
\end{equation}
for any sheaf $\F$ on $X'$ and any $x\in X$.
In case $x=x'\in X'$, the latter becomes 
\begin{equation}\label{35}
\F(X'_{x'})=\F_{x'}.
\end{equation}

\begin{lem}\label{9}\
\begin{asparaenum}[(a)] 
\item\label{9a} Let $\M$ be a coherent $\O_X$-module. 
Then
\begin{equation}\label{36}
\M\otimes_{\O_X}\Q_X=i_*i^{-1}\M=\left(U\mapsto\varprojlim_{x\in X'\cap U}\M_x\right).
\end{equation}
Note that $\varprojlim_{X'\cap U}=\prod_{X'\cap U}$ if $X$ has no embedded points.
\item\label{9c} For any affine open $U=\Spec A\subset X$ and $M:=\Gamma(X,\M)$,
\begin{equation}\label{50}
\Gamma(U,\M\otimes_{\O_X}\Q_X)=M\otimes_AQ(A).
\end{equation}
\item\label{9b} We have $\M_x\otimes_{\O_{X,x}}\Q_{X,x}=\varprojlim_{x'\in X'_x}\M_{x'}$ for any $x\in X$.
\end{asparaenum}
\end{lem}

\begin{proof}
Setting $\M'(V):=\varprojlim_{x'\in V}\M_{x'}$ for any open $V\subset X'$ defines a sheaf on $X'$ (see \cite[\S4.2.2]{Cur14}).
For any open $U\subset X$,
\begin{equation}\label{29}
i_*\M'(U)=\varprojlim_{x\in X'\cap U}\M_x.
\end{equation}
We may therefore read the right-hand sheaf in \eqref{36} as $i_*\M'$ and the right-hand of \eqref{9b} as
$\M'(X'_x)$.
Now \eqref{28} reduces \eqref{9b} to proving \eqref{9a}.

By \eqref{30} and \eqref{29}, we settle \eqref{9a} in case $\M=\O_X$, by proving that
\begin{equation}\label{31}
i^{-1}\M=\M'.
\end{equation}
There is a natural morphism of sheaves $\M\to i_*\M'$.
Since $i^{-1}$ is left-adjoint to $i_*$, this gives rise to a morphism $i^{-1}\M\to\M'$.
That it is an isomorphism can be checked stalk-wise at any $x'\in X'$ using \eqref{35}:
\begin{equation}\label{40}
(i^{-1}\M)_{x'}=\M_{x'}=\varprojlim_{X'\ni y'\leadsto x'}\M_{y'}=\M'(X'_{x'})=\M'_{x'}.
\end{equation}
Since $i^{-1}$ is left-adjoint to $i_*$, the identity morphism of $i^{-1}\M$ induces a natural morphism $\M\to i_*i^{-1}\M$.
Using \eqref{30} and \eqref{31}, this yields a natural morphism of sheaves
\[
\M\otimes_{\O_X}\Q_X\to i_*\M'.
\]
To establish both \eqref{9a} and \eqref{9c} we show that this induces an isomorphism of global sections over any affine open using the presheaf tensor product.
To this end, we assume that $X=\Spec A$ and set $M:=\Gamma(X,\M)$.
Then, using Lemma~\ref{7} and the claim in case $\M=\O_X$, it suffices to show that
\begin{equation}\label{32}
M\otimes_A\varprojlim_{\pp\in\Ass A}A_\pp\cong\varprojlim_{\pp\in\Ass A}M_\pp.
\end{equation}
By $\O_X$-coherence of $\M$, $M$ is a finitely presented $A$-module.
Since $A$ is Noetherian, $\Ass A$ is finite and hence the Mittag--Leffler condition is trivially satisfied.
Therefore, the inverse limit commutes with tensor product and \eqref{32} holds true.
\end{proof}

\section{Rank of coherent modules}

Recall that an $A$-module $M$ has rank $\rk M=\rk_A M=r$ if $M\otimes_AQ(A)\cong Q(A)^r$ (see \cite[Def.~1.4.2]{BH93}).
In case $M$ is finite, this is equivalent to $M_\pp\cong A_\pp^r$ for all $\pp\in\Ass A$ (see \cite[Prop.~1.4.3]{BH93}).

\begin{dfn}\label{51}
Let $\M$ be a coherent $\O_X$-module.
We say that $\M$ has \emph{global rank} $\rk\M=\rk_X\M=r$ if
\begin{equation}\label{24}
\M\otimes_{\O_X}\Q_X\cong\Q_X^r.
\end{equation}
We say that $\M$ has \emph{local rank} $\rk\M=\rk_X\M=r$ if $\M_{x'}\cong\O_{X,x'}^r$ for all $x'\in X'$.
In case $\M\into\Q_X^r$, we say that $\M$ has a \emph{rank} if it has a local, or equivalently global, rank (see Lemma~\ref{23} below).
\end{dfn}

The following easy result applies to $A=Q(A)$.

\begin{lem}\label{26}
Let $A$ be a ring in which all regular elements are units.
Then any inclusion of free modules of equal finite rank is an equality.\qed
\end{lem}

\begin{lem}\label{23}
Let $\M$ be a coherent $\O_X$-module.

\begin{enumerate}[(a)] 

\item\label{23c} If $\M$ has a global rank, then $\M$ has the same local rank.

\item\label{23d} $\M$ has a local rank $\rk\M=r$ if and only if $\M_{X,x}$ has rank $\rk\M_{X,x}=r$ for all $x\in X$.

\item\label{23a} If $X=\Spec A$ is affine and $\M=\wt M$, then the following are equivalent.
\begin{enumerate}[(1)]
\item\label{23a2} $\M$ has global rank $\rk\M=r$.
\item\label{23a0} $\M$ has local rank $\rk\M=r$.
\item\label{23a1} $M$ has rank $\rk M=r$.
\end{enumerate}

\item\label{23b} If $\M\into\Q_X^r$, then the following are equivalent.
\begin{enumerate}[(1)]
\item\label{23b0} $\M$ has local rank $\rk\M=r$.
\item\label{23b1} $\M$ has local rank $\rk\M\ge r$.
\item\label{23b2} The induced morphism $\M\otimes_{\O_X}\Q_X\to\Q_X^r$ is an isomorphism.
\item\label{23b3} For any affine open $U=\Spec A\subset X$ and $M:=\Gamma(U,\M)$, the induced morphism $M\otimes_AQ(A)\to Q(A)^r$ is an isomorphism.
\end{enumerate}
In particular, a local rank is global in this case.

\end{enumerate}
\end{lem}

\begin{proof}\pushQED{\qed}
By coherence of $\M$, $M$ is finite.
\begin{asparaenum}[(a)] 

\item Taking stalks at $x\in X$ in \eqref{24} this follows from Lemma~\ref{4}.\eqref{4a}.

\item This follows from Lemma~\ref{41}.

\item By \eqref{23c}, \eqref{23a2} implies \eqref{23a0}.
By Lemma~\ref{7}, points $x\in X'$ correspond one-to-one to prime ideals $\pp\in\Ass A$.
Then $\M_x=M_\pp$ and $\O_{X,x}=A_\pp$ and \eqref{23a0} implies \eqref{23a1} by \cite[Prop.~1.4.3]{BH93}.
Assuming the latter, loc.~cit.~gives a short exact sequence of $A$-modules
\[
0\to F\to M\to T\to 0
\]
where $F$ is free of rank $r$ and $T$ is torsion.
These properties are preserved under localization.
By Lemmas~\ref{4}.\eqref{4a} and \ref{26}, applying $\wt{-}\otimes_{\O_X}\Q_X$ turns it into an isomorphism $\Q_X^r\cong\M\otimes_{\O_X}\Q_X$.
Thus, \eqref{23a1} implies \eqref{23a2}.

\item By \eqref{23a}, \eqref{23b3} implies \eqref{23b0}, which trivially implies \eqref{23b1}.
By Lemma~\ref{9}.\eqref{9a}, the morphism in \eqref{23b2} reads
\[
\M\otimes_{\O_X}\Q_X=i_*i^{-1}\M\to i_*i^{-1}\O_X^r=\Q_X^r
\]
and is induced by $i^{-1}\M\to i^{-1}\O_X^r$.
Its stalk at $x'\in X'$ is the inclusion $\M_{x'}\into\O_{X,x'}^r$ from the hypothesis.
If \eqref{23b1} holds it must be an equality by Lemma~\ref{26} and \eqref{23b2} follows.

For $U$ and $M$ as in \eqref{23b3}, by Lemma~\ref{4}.\eqref{4b} and injectivity of sheafification on sections, \eqref{23b2} yields an inclusion $M\otimes_AQ(A)\into\Gamma(U,\M\otimes_{\O_X}\Q_X)=Q(A)^r$.
For it to be an isomorphism it suffices to show that $\rk M=r$ by Lemma~\ref{26}.
By \cite[Prop.~1.4.3]{BH93} and Lemma~\ref{7}, this is equivalent to $\M_{x'}\cong\O_{X,x'}^r$ for all $x'\in U\cap X'$.
This follows from \eqref{23b2} due to Lemma~\ref{4}.\eqref{4a}.
Alternatively one could use that $\M_U\otimes_{\O_U}\Q_U$ is the presheaf tensor product as observed in the proof of Lemma~\ref{9} .
\qedhere

\end{asparaenum}
\end{proof}

\section{Fractional morphisms}

\begin{dfn}\label{2}
Let $\I$ be a coherent $\O_X$-submodule of $\Q_X$ and let $f:Y\to X$ be a morphism of schemes.
\begin{enumerate}[(i)] 
\item\label{2a} We call $\I$ a \emph{fractional ideal on $X$} if it has a rank.
\item\label{2b} We call $f$ a \emph{fractional morphism} if it induces a morphism 
\begin{equation}\label{44}
\xymat{
\O_Y\ar@{^(->}[r] & \Q_Y\\
f^{-1}\O_X\ar@{^(->}[r]\ar[u]^-{f^\#} & f^{-1}\Q_X\ar@{-->}[u]
}
\end{equation}
We call it a \emph{bifractional} if this morphism induces an isomorphism $\Q_X\cong f_*\Q_Y$.
\item\label{2c} For a fractional $f$, the \emph{inverse image of $\I$ under $f$} is the $\O_X$-submodule 
\[
\O_Y\I:=\O_Yf^\#(f^{-1}\I)\subset\Q_Y.
\]
\end{enumerate}
\end{dfn}

Due to Lemma~\ref{23} fractionality of ideals is an affine local property.

\begin{lem}\label{25}
Let $\I\subset\Q_X$ be a quasicoherent $\O_X$-submodule.
Then $\I$ is a fractional ideal on $X$ if and only if $I=\Gamma(U,\I)$ is a fractional ideal of $A$ for each affine open $U=\Spec A\subset X$.\qed
\end{lem}

\begin{rmk}\label{11}
Let $\I\subset\Q_X$ be a quasicoherent $\O_X$-submodule.
\begin{enumerate}[(a)]
\item\label{11b} By left exactness of the section functor and Lemma~\ref{4}.\eqref{4b}, $\I$ is coherent if and only if (affine) locally $\alpha\I\subset\O_X$ for some regular $\alpha\in\O_X$.
\item\label{11c} By Lemma~\ref{23}.\eqref{23b}, any fractional ideal $\I\ne0$ has (global) rank $\rk\I=1$ which means that (affine) locally $\alpha\O_X\subset\I$ for some regular $\alpha\in\O_X$.
\item\label{11a} By Lemma~\ref{4}.\eqref{4b}, $f:Y\to X$ is a fractional morphism if and only if for any restriction $\Spec B\to\Spec A$ of $f$ to affine open subsets, $B$ is a torsion free $A$-module.
\item\label{11d} Setting $\I=\O_X$ in Definition~\ref{2}.\eqref{2c}, $\O_Y\O_X=f^*\O_X=\O_Y$.
\item\label{11e} If $\alpha$ is as in \eqref{11b} or \eqref{11c} for $\I$, then $f^\#(\alpha)$ has the corresponding property for $\O_Y\I$.
In particular, $\O_Y\I$ inherits coherence and fractionality from $\I$.
\end{enumerate}
\end{rmk}

\section{Associating morphisms}

The notion of a fractional morphism in Defintion~\ref{2}.\eqref{2b} can be expressed in terms of associated points.
To this end we consider a variation of the notions of dominant and birational morphism (see \cite[\href{http://stacks.math.columbia.edu/tag/01RJ}{Def.~01RJ}, \href{http://stacks.math.columbia.edu/tag/01RO}{Def.~01RO}]{stacks-project} for comparison).

\begin{dfn}\label{8}
Let $f:Y\to X$ be a quasicompact morphism of locally Noetherian schemes.
\begin{enumerate}[(i)]
\item\label{8a} We call $f$ \emph{associating} if it induces a map $\Ass Y\to\Ass X$.
\item\label{8b} We call $f$ \emph{biassociating} if it induces a bijection $\Ass Y\to\Ass X$ 
and an isomorphism $f^\#_{\wt y}:\O_{X,f(\wt y)}\to\O_{Y,\wt y}$ for any maximal/closed $\wt y\in\Ass Y$.
\end{enumerate}
\end{dfn}

\begin{lem}\label{5}\
\begin{enumerate}[(a)] 
\item\label{5a} If $f:Y\to X$ is associating, then it is fractional.
\item\label{5b} The converse statement of \eqref{5a} holds if $X$ has no embedded points.
\item\label{5c} With $X$ also $Y$ has no embedded points if $f:\Ass Y\to\Ass X$ is a bijection.
\item\label{5d} If $f:Y\to X$ is biassociating, then it is bifractional and induces a homeomorphism $\Ass Y\to\Ass X$.
\end{enumerate}
\end{lem}

\begin{proof}\pushQED{\qed}\
\begin{asparaenum}[(a)]

\item Suppose that $f:Y\to X$ is associating.
Then there is a commutative diagram
\begin{equation}\label{19}
\xymat{
Y'=\Ass Y\ar[d]_g\ar@{^(->}[r]_-j & Y\ar[d]^f \\
X'=\Ass X\ar@{^(->}[r]_-i & X
}
\end{equation}
where $g$ is a map of posets.
Applying $j^{-1}$ to $f^\#:f^{-1}\O_X\to\O_Y$, this yields a morphism
\[
j^{-1}f^\#:g^{-1}i^{-1}\O_X=j^{-1}f^{-1}\O_X\to j^{-1}\O_Y.
\]
Applying $j_*$ and composing with the natural transformation $f^{-1}i_*\to j_*g^{-1}$ applied to $i^{-1}\O_X$ yields the desired morphism
\[
\xymat{
f^{-1}\Q_X=f^{-1}i_*i^{-1}\O_X\ar[r] & j_*g^{-1}i^{-1}\O_X\ar[r]^-{j_*j^{-1}f^\#} & j_*j^{-1}\O_Y=\Q_Y.
}
\]

To see the above natural transformation, let $\F$ be a sheaf on $X'$ and $U\subset Y$ open.
Since sheafification is left adjoint to the forgetful functor from sheaves to presheaves, we may use the presheaf inverse image.
Then
\[
(f^{-1}i_*\F)(U)=\varinjlim_{f(U)\subset V}\F(i^{-1}(V)),\quad (j_*g^{-1}\F)(U)=\varinjlim_{g(j^{-1}(U))\subset W}\F(W).
\]
By the commutative diagram \eqref{19}, $i^{-1}(V)$ is a valid choice for $W$.
Indeed, $f(U)\subset V$ means $U\subset f^{-1}(V)$.
Applying $j^{-1}$ this gives $j^{-1}(U)\subset j^{-1}f^{-1}(V)=g^{-1}i^{-1}(V)$, and finally $g(j^{-1}(W))\subset i^{-1}(V)$ as claimed.

\item Suppose conversely that $f$ induces a morphism $f^{-1}\Q_X\to\Q_Y$.
Let $y\in\Ass Y$.
By Lemma~\ref{7}, we may assume that $X=\Spec A$ and $Y=\Spec B$ with $y=\qq\in\Ass B$.
By Lemma~\ref{4}.\eqref{4b}, the hypothesis then becomes that $f^\#$ induces a morphism $Q(A)\to Q(B)$.
Assume that $x=f(y)\not\in\Ass X$.
Then $x=\pp=(f^\#)^{-1}(\qq)\not\in\Ass A$ and hence $\pp\not\subset\pp'$ for any $\pp'\in\Ass A$ as $X$ has no embedded points by hypothesis.
Then prime avoidance yields an $\alpha\in\pp\setminus\bigcup\Ass A$.
This means that $\alpha\in A$ is regular while $f^\#(\alpha)\in\qq\in\Ass B$ is not, a contradiction.
It follows that $x\in\Ass X$ and hence that $f$ is associating.

\item Since $f$ is continuous it is order preserving and the claim follows.

\item Let $y\in Y'$ and $x=g(y)\in X'$.
Pick $\bar x\in X'$ maximal such that $x\le\bar x$.
By the first part of Definition~\ref{8}.\eqref{8b} there is a unique $\bar y\in Y'$ with $g(\bar y)=\bar x$, which is then maximal by continuity of $g$.
By its second part, $g^\#_{\bar y}:\O_{X,\bar x}\to\O_{Y,\bar y}$ is an isomorphism.
Due to \eqref{6} $y\le\bar y$ and $g^{-1}(X'_x)=Y'_y$.
In particular, $g$ is a homeomorphism.
Localizing $g^\#_{\bar y}$ at the prime of $\O_{X,\bar x}$ corresponding to $x$ via \eqref{6}, \eqref{40} yields an isomorphism
\[
g^\#_y:\O'_X(X'_x)=\O_{X,x}\to\O_{Y,y}=\O'_Y(Y'_y)=(g_*\O'_Y)(X'_x).
\]
Since the $X'_x$ form a basis of the topology of $X$, it follows that $g^\#:\O'_X\to g_*\O'_Y$ is an isomorphism.
By \eqref{30}, \eqref{31} and commutativity of \eqref{19}, applying $i_*$ yields the desired isomorphism
\[
\xymat{
f^\#:\Q_X=i_*\O'_X\ar[r]^-{i_*g^\#}_-\cong & i_*g_*\O'_Y=f_*j_*\O'_Y=f_*\Q_Y.
}\qedhere
\]

\end{asparaenum}
\end{proof}

\section{Blowup of fractional ideals}

\begin{dfn}
Let $X$ be a scheme and let $\I$ be a coherent $\O_X$-submodule of $\Q_X$.
Then the \emph{blow up of $X$ along $\I$} is defined as
\begin{equation}\label{1}
\xymat{
Z:=\Bl_\I X:=\Proj_X\bigoplus_{i\ge0}\I^i\ar[r]^-{b} & X.
}
\end{equation}
\end{dfn}

\begin{lem}\label{10}
The blow up morphism $b$ in \eqref{1} is fractional.
\end{lem}

\begin{proof}
Using Remark~\ref{11}.\eqref{11a}, we may assume that $X=\Spec A$, $\I=\wt I$ for some $I\subset Q(A)$, and replace $Z$ by an affine open $V=\Spec B$ where $B=(\bigoplus_{i\ge0}I^i)_{(g)}$ for some $g\in I^i$ where $i\ge 1$.
We have to show that $B$ is a torsion free $A$-module.
Any $a\in A^\reg$ is a unit in $Q(A)$ and hence regular on $\bigoplus_{i\ge0}I^i$ since $I^i\subset Q(A)$.
Since localization and projection to a direct summand are exact operations, $a$ is also regular on $B$.
\end{proof}

\begin{rmk}\label{45}\
\begin{asparaenum}[(a)]

\item\label{45a} Denote by $\A:=\bigoplus_{i\ge0}\I^i$ the graded $\O_X$-algebra defining $\Bl_\I X$.
The inverse image $\O_Z\I=\O_Z(1)$ is the sheaf associated to the graded $\A$-module
\[
\I\bigoplus_{i\ge0}\I^i=\bigoplus_{i\ge0}\I^{i+1}.
\]
It is invertible since the $\A$ is generated by $\A_1=\I$.
In case $\I$ is invertible, $\Bl_\I X=X$.

\item\label{45b} By Remark~\ref{11}.\eqref{11b}, there exists locally a unit $\alpha$ in $\Q_X$ such that $\alpha\I$ is an $\O_X$-ideal.
Then $Z=\Bl_\I X$ is locally isomorphic to the blow up $Z':=\Bl_{\alpha\I}X$ of $X$ along $\alpha\I$ since multiplication by $\alpha^i$ in degree $i$ induces a graded isomorphism of $\O_X$-algebras 
\[
\xymat{
\bigoplus_{i\ge0}\I^i\ar[r]_-\cong^-{\alpha^\bullet\cdot} & \bigoplus_{i\ge0}\alpha^i\I^i.
}
\]
Over it, there is a graded isomorphism
\[
\xymat{
\bigoplus_{i\ge0}\I^{i+1}\ar[r]_-\cong^-{\alpha^\bullet\cdot} & \frac1\alpha\bigoplus_{i\ge0}(\alpha\I)^{i+1}
}
\]
which shows that $\O_Z(1)$ is isomorphic $\O_{Z'}(1)$ via the local isomorphism $Z\cong Z'$.

\end{asparaenum}
\end{rmk}

The universal property of blow up \cite[Def.~IV-16]{EH00} extends to blow ups of fractional ideals as follows.
This was remarked in \cite[Prop.~2.1]{OZ91} under slightly stronger hypotheses.

\begin{prp}\label{17}
Let $X$ be a scheme and let $\I$ be a coherent $\O_X$-submodule of $\Q_X$.
In the full subcategory of fractional morphisms $f:Y\to X$ such that $\O_Y\I$ is invertible, the blow up \eqref{1} is a terminal object.
\end{prp}

\begin{proof}
Let $f:Y\to X$ be a morphism as in the statement.
Set $\L:=\O_Y\I$ such that $f^*\I\onto\L$.
Then
\[
f^*\bigoplus_{i\ge0}\I^i\to\bigoplus_{i\ge0}\O_Y\I^i=\bigoplus_{i\ge0}\L^{\otimes i}
\]
is a morphism of graded $\O_X$-algebras.
By the universal property of $\Proj_X$ (see \cite[\href{http://stacks.math.columbia.edu/tag/01O4}{Lem.~01O4}]{stacks-project}), this induces a morphism $Y\to Z$ of schemes over $X$ as required. 
\end{proof}

\section{Blowup of conductors}\label{14}

\begin{dfn}\label{58}\
\begin{enumerate}[(a)]

\item Let $f:Y\to X$ be a finite morphism of Noetherian schemes and let $\F$ be a coherent $\O_X$-module.
Then $\SHom_{\O_X}(f_*\O_Y,\F)$ defines a coherent $\O_Y$-module $f^!\F$ (see \cite[Prop.~6.4.25.(a)]{Liu02}).
The $\O_X$-ideal $\C_{Y/X}:=\SHom_{\O_X}(f_*\O_Y,\O_X)$ is the \emph{conductor of $f$}.

\item Let $X$ be a reduced scheme and let $\ol\O_X$ be the integral closure of $\O_X$ in $\Q_X$.
Then $f:\wt X:=\Spec_X\ol\O_X\to X$ is the \emph{normalization} and $\C_{\wt X/X}$ is the \emph{conductor of $X$}.

\end{enumerate}
\end{dfn}

\begin{rmk}\label{18}
A scheme $X$ is called \emph{Nagata} if for any affine open $U=\Spec A\subset X$ (in some cover) the ring $A$ is Nagata.
Nagata schemes have a finite normalization morphism.
Since we need only this latter property, we omit the details of the definition.
\end{rmk}

To define conductors of fractional ideals we apply the following easy result.

\begin{lem}\label{20}
Let $A$ be a ring and let $J$ and $I$ be fractional ideals of $A$. 
Then $\Hom_A(J,I)\cong I:_{Q(A)}J$ given by $\frac{\varphi(a)}{a}\leftrightarrow\varphi$ for any $a\in A^\reg\cap J$.
In particular, $\Hom_A(J,I)$ is again a fractional ideal.\qed
\end{lem}

\begin{lem}\label{21}
Let $f:Y\to X$ be a finite bifractional morphism of schemes.
Let $\J$ and $\I$ be fractionals on $Y$ and $X$ respectively.
Then $f_*\J$ and 
\[
\C_{\J/\I}:=\SHom_{\O_X}(f_*\J,\I)=\I:_{\Q_X}f_*\J
\]
are fractional ideals on $X$ and there is a unique fractional ideal $\C'_{\J/\I}$ on $Y$ such that
\[
f_*(\C'_{\J/\I})=\C_{\J/\I}.
\]
\end{lem}

\begin{proof}
Since $f$ is finite and hence affine, $f_*\J$ is coherent and we may assume that both $X=\Spec(A)$ and $Y=\Spec(B)$ are affine.
By Lemma~\ref{4}.\eqref{4b}, \eqref{44} up to isomorphism translates to 
\begin{equation}\label{43}
\xymat{
B\ar@{^(->}[r] & Q(B)\\
A\ar@{^(->}[r]\ar@{^(->}[u]^-{\varphi} & Q(A)\ar@{=}[u]
}
\end{equation}
By Lemma~\ref{25}, $J:=\Gamma(Y,\J)=\Gamma(X,f_*\J)$ and $I:=\Gamma(X,\I)$ are fractional ideals of $B$ and $A$ respectively, and $\SHom_{\O_X}(f_*\J,\I)$ is the sheaf associated to $\Hom_A(J,I)$ where $J$ is an $A$-module via $\varphi$.
We may assume that $I\ne0\ne J$.
Since $\varphi:A\to B$ is finite, $J$ is a finite $A$-module and $J\otimes_AQ(A)=J\otimes_B Q(B)\cong Q(B)=Q(A)$ which means that $\rk_AJ=\rk_BJ=\rk_Y\J=1$ by Lemma~\ref{23}.\eqref{23a} and Remark~\ref{11}.\eqref{11c}.
Hence $J$ and then by Lemma~\ref{20} also $\Hom_A(J,I)$ is a non zero fractional ideal of $A$.
With $J$ also $\Hom_A(J,I)$ is a $B$-module which is finite over $A$ and hence over $B$.
Thus, $\Hom_A(J,I)\otimes_BQ(B)=\Hom_A(J,I)\otimes_AQ(A)\cong Q(A)$ and therefore $\rk_B\Hom_A(J,I)=\rk_A\Hom_A(J,I)=1$.
Thus, using Lemma~\ref{25}, the fractional ideal
\begin{equation}\label{33}
\C'_{\J/\I}:=\wt{\Hom_A(J,I)}
\end{equation}
is the $\O_Y$-module associated to the $B$-module $\Hom_A(J,I)$.
\end{proof}

\begin{rmk}\label{62}
If $f:Y\to X$ is finite bifractional, then $\C'_{\O_Y/\I}=f^!\I$ in Lemma~\ref{21}.
\end{rmk}

\begin{lem}\label{13}
Under the hypotheses of Lemma~\ref{21}, 
$\Bl_{\C_{\J/\I}}X=\Bl_{\C'_{\J/\I}}Y$.
\end{lem}

\begin{proof}
Reduce to an affine situation as in Lemma~\ref{21}.
Setting $C:=\Gamma(X,\C_{\J/\I})=\Gamma(Y,\C'_{\J/\I})$, we have 
\[
\Bl_{\C_{\J/\I}} X=\Proj_A(A\oplus C\oplus C^2\oplus\cdots)=\Proj_B(B\oplus C\oplus C^2\oplus\cdots)=\Bl_{\C'_{\J/\I}}Y
\]
which proves the claim.
\end{proof}

Combining Remark~\ref{45}.\eqref{45a} and Lemma~\ref{13} immediately implies the following

\begin{thm}\label{22}
Let $f:Y\to X$ be a finite bifractional morphism of schemes.
Let $\J$ and $\I$ be fractional on $Y$ and $X$ respectively.
Then $\Bl_{\C_{\J/\I}} X=Y$ if and only if $\C'_{\J/\I}$ is invertible.\qed
\end{thm}

Using Remark~\ref{18}, we recover a result of P.M.H.~Wilson~\cite[Cor.~1.4]{Wil78} with reduced hypotheses.

\begin{cor}
Let $X$ be a reduced one-dimensional locally Noetherian scheme with finite normalization $\wt X\to X$.
Then $\Bl_{\C_{\wt X/X}} X=\wt X$.
\end{cor}

\begin{proof}
By hypothesis $\O_{\wt X}$ a sheaf of PIDs and hence any fractional ideal on $\wt X$ is invertible.
Then the claim follows from Theorem~\ref{22}.
\end{proof}

In this context it is interesting to know when $\C'_{\J/\I}$ is reflexive.

\begin{prp}
In addition to the hypotheses of Theorem~\ref{22}, assume that $Y$ is $(S_2)$ and Gorenstein in codimension up to one.
If $\I$ is $(S_2)$, then $\C'_{\J/\I}$ is reflexive.
\end{prp}

\begin{proof}\pushQED{\qed}
Reduce to an affine situation as in Lemma~\ref{21}.
By \eqref{33} and \cite[Thm.~3.6]{EG85}, it suffices to show that $\Hom_A(J,I)_\qq$ is an $(S_2)$ $B_\qq$-module for any $\qq\in\Spec B$.
Setting $\pp:=\varphi^{-1}(\qq)\in\Spec A$, we may replace $A$ by $A_\pp$ with maximal ideal $\mm:=\pp A_\pp$, $\qq$ by $\nn:=\qq B_\pp$, and assume that $A$ is local.
By finiteness of $B$ over $A$, $\dim A\ge\dim B\ge\dim B_{\qq}$.
Using \cite[Prop.~1.2.10.(a), Exc.~1.4.19]{BH93} and that $I$ is $(S_2)$, this implies
\begin{align*}
\depth_{B_\qq}\Hom_A(J,I)_\qq
&\ge\grade(\mm,\Hom_A(J,I))
=\depth_A\Hom_A(J,I)\\
&\ge\min\{2,\depth_AI\}
=\min\{2,\dim A\}\\
&\ge\min\{2,\dim B\}\ge\min\{2,\dim B_{\qq}\}.\qedhere
\end{align*}
\end{proof}

\section{Blowup of canonical ideals}\label{42}

Let $X$ be a \emph{Cohen--Macaulay scheme}.
That is, $X$ is locally Noetherian and $\O_{X,x}$ is a Cohen--Macaulay ring for all (closed) points $x\in X$ (see \cite[\href{http://stacks.math.columbia.edu/tag/02IP}{Lem.~02IP}]{stacks-project}).
The condition on closed points suffices due to Lemma~\ref{56} and since the Cohen--Macaulay property localizes (see \cite[Thm.~2.1.3.(b)]{BH93}).
We use an analogous definition for a \emph{Gorenstein scheme} (see \cite[\href{http://stacks.math.columbia.edu/tag/0AWW}{Def.~0AWW}]{stacks-project}).
By a \emph{canonical module} $\omega_X$ on a Cohen--Macaulay scheme $X$ we mean a coherent $\O_X$-module such that $\omega_{X,x}$ is a canonical module for $\O_{X,x}$ in the sense of \cite[Def.~3.3.1]{BH93} for all (closed) points $x\in X$.
Recall that by \cite[Thm.3.3.5.(b)]{BH93} the canonical module localizes.
If $\omega_X$ is isomorphic to a fractional ideal we call it a \emph{canonical ideal}.

\begin{lem}\label{52}
Let $X$ be a Cohen--Macaulay scheme with canonical module $\omega_X$.
Assume that $\omega_X$ has a global rank.
Then $\omega_X$ is a canonical ideal.
\end{lem}

\begin{proof}
By Lemma~\ref{23}.\eqref{23c} and \cite[Prop.~3.3.18]{BH93}, $\rk\omega_{X,x}=1$, that is, $X$ is generically Gorenstein.
Since $\omega_{X,x}$ is a maximal Cohen--Macaulay module it is torsion free and hence taking stalks of the canonical morphism $\omega_X\to\omega_X\otimes_{\O_X}\Q_X\cong\Q_X$ yields the claim.
\end{proof}

Let $X$ be a Cohen--Macaulay scheme with canonical ideal $\omega_X$.
Let $f:Y\to X$ be a finite bifractional morphism of schemes.
We abbreviate 
\[
\omega_Y:=f^!\omega_X.
\]
By Lemma~\ref{21} and Remark~\ref{62}, $\SHom_{\O_X}(f_*\O_Y,\omega_X)$ and $\omega_Y$ are fractional ideals on $X$ and $Y$ respectively and related by
\begin{equation}\label{15}
f_*\omega_Y=\SHom_{\O_X}(f_*\O_Y,\omega_X).
\end{equation}

\begin{dfn}
We say that a morphism $f:Y\to X$ is \emph{equidimensional along fibers of closed points} if, for all closed points $x\in X$, $\dim\O_{Y,y}$ is independent of $y\in f^{-1}(x)$.
\end{dfn}

\begin{prp}\label{12}
Let $f:Y\to X$ be a finite bifractional morphism of Cohen--Macaulay schemes which is equidimensional along fibers of closed points and let $\omega_X$ be a canonical ideal on $X$.
Then $\omega_Y$ is invertible if and only if $Y$ is Gorenstein.
\end{prp}

\begin{proof}
Since $\omega_Y$ is coherent, it is invertible if and only if $\omega_{Y,y}\cong\O_{Y,y}$ for all closed $y\in Y$.
By faithful flatness of completion, this is equivalent to  $\wh{\omega_{Y,y}}\cong\wh{\O_{Y,y}}$ for all closed $y\in Y$.
By \cite[Prop.~3.1.19, Thm.~3.3.7]{BH93} and Lemma~\ref{16} below, this is equivalent to $Y$ being Gorenstein.
\end{proof}

\begin{lem}\label{57}
Let $f:Y\to X$ be a finite bifractional morphism of schemes.
Then $y\in Y$ is closed if and only $x=f(y)\in X$ is closed.
\end{lem}

\begin{proof}
By hypothesis $f$ corresponds affine locally in the target to an integral extension $\varphi$ as in \eqref{43}.
Here the going up and incomparability theorem apply and Lemma~\ref{56} implies the claim.
\end{proof}

\begin{lem}\label{16}
Under the hypotheses of Proposition~\ref{12}, $\wh{\omega_{Y,y}}$ is a canonical module for $\wh{\O_{Y,y}}$ for all closed $y\in Y$.
\end{lem}

\begin{proof}
Let $y\in Y$ closed.
Then $x=f(y)\in X$ is closed by Lemma~\ref{57}.
By hypothesis $f$ is finite hence quasifinite and affine.
Then $x$ has finite preimage $f^{-1}(x)=\{y=y_1,y_2,\dots,y_n\}$.
We may assume that both $Y=\Spec B$ and $X=\Spec A$ are affine and identify $x=\pp\in\Spec A$, $y=\qq,y_i=\qq_i\in\Spec B$.
Since $f$ is finite bifractional, $\varphi:A\into B$ in \eqref{43} and hence $\varphi_\pp:A_\pp\into B_\pp$ is a finite extension.
Setting $\omega_A=\Gamma(X,\omega_X)$ and $\omega_B=\Gamma(Y,\omega_Y)$, \eqref{15} reads 
$\omega_B=\Hom_A(B,\omega_A)$ and hence $\omega_{B,\pp}=\Hom_{A_\pp}(B_\pp,\omega_{A_\pp})$.
We may therefore assume that $A=(A,\pp)$ is local with canonical module $\omega_A$.

Since $B$ is integral over $A$, $\dim A=\dim B\ge\dim B_{\qq_i}$ with equality for some $i$.
Assuming $y$ closed means that $\qq\lhd B$ is maximal which is equivalent to $\pp\lhd A$ being maximal by \cite[Cor.~5.8]{AM69}.
It follows that $B$ is semilocal with maximal ideals $\qq_1,\dots,\qq_n\lhd B$.
Using the equidimensionality hypothesis it follows that $\dim B_{\qq_i}=\dim A$ for all $i$.
The ideal $\pp B$ defines the same topology as the Jacobson radical $\bigcap_{i=1}^n\qq_i$ and hence $B\otimes_A\wh A=\wh B$.
By \cite[Thm.~8.15]{Mat89}, there is a product decomposition
\begin{equation}\label{60}
\wh B=\prod_{i=1}^n\wh{B_{\qq_i}}.
\end{equation}
Then each $\wh{B_{\qq_i}}$ is a finite $\wh A$-module of dimension $\dim\wh{B_{\qq_i}}=\dim \wh A$.

Since $B$ is finitely presented and $\wh A$ is flat over $A$, \cite[Thm.~3.3.5.(c), Thm.~3.3.7.(b)]{BH93} yields that
\begin{align*}
\omega_B\otimes_B\wh B
&=\Hom_A(B,\omega_A)\otimes_B\wh B
=\Hom_A(B,\omega_A)\otimes_A\wh A
=\Hom_{\wh A}(B\otimes_A\wh A,\omega_A\otimes_A\wh A)\\
&=\Hom_{\wh A}(\prod_{i=1}^n\wh{B_{\qq_i}},\omega_{\wh A})
=\bigoplus_{i=1}^n\Hom_{\wh A}(\wh{B_{\qq_i}},\omega_{\wh A})
=\bigoplus_{i=1}^n\omega_{\wh{B_{\qq_i}}}.
\end{align*}
The claim then follows by applying $-\otimes_{\wh B}\wh{B_\qq}$:
\[
\wh{\omega_{Y,y}}=\wh{\omega_{B,\qq_1}}=\omega_B\otimes_BB_\qq\otimes_{B_\qq}\wh{B_\qq}=\omega_B\otimes_B\wh{B_\qq}=\omega_B\otimes_B\wh B\otimes_{\wh B}\wh{B_\qq}=\omega_{\wh{B_\qq}}=\omega_{\wh{\O_{Y,y}}}.
\]
\end{proof}

We now generalize a result of P.M.H.~Wilson~\cite[Thm.~2.7, Rem.~2.8]{Wil78}.

\begin{thm}\label{47}
Let $f:Y\to X$ be a finite bifractional morphism of schemes.
Assume that $X$ is Cohen--Macaulay with canonical ideal $\omega_X$.
Then $\Bl_{f_*f^!\omega_X} X=Y$ if and only if $f^!\omega_X$ is invertible. 
The latter is equivalent to $Y$ being Gorenstein if $Y$ is Cohen--Macaulay and $f$ is equidimensional along fibers of closed points.
\end{thm}

\begin{proof}
This follows from Theorem~\ref{22} and Proposition~\ref{12}.
\end{proof}

Taking Remark~\ref{18} into account, Theorem~\ref{48} in \S\ref{53} is now a consequence of Theorem~\ref{47} and the following result.

\begin{prp}\label{54}
Let $X$ be a reduced Cohen--Macaulay scheme.
If $f:Y\to X$ is finite bifractional, then it is equidimensional along fibers of closed points.
\end{prp}

\begin{proof}
Let $x\in X$ be a closed point and $y\in f^{-1}(x)$.
We may return to the affine local situation of the proof of Lemma~\ref{16}. 
Since $A$ is reduced we have $Q(A)_\pp=Q(A_\pp)$.
By exactness of localization, we may assume that $A=(A,\pp)$ is local Cohen--Macaulay and
\begin{equation}\label{61}
\xymat{
A\ar@{^(->}[r]^-\varphi & B\ar@{^(->}[r] & Q(A)
}
\end{equation}
with $\varphi$ a finite extension.
By exactness of completion, $-\otimes_A\wh A$ preserves injectivity and finiteness in \eqref{61} and $Q(A)\otimes_A\wh A\into Q(\wh A)$.
Since the Cohen--Macaulay property commutes with completion (\cite[Cor.~2.1.8]{BH93}), we may assume that $A$ is complete and that $B$ is decomposed as in \eqref{60}.
In particular, $B$ is catenary (see \cite[Thm.~2.1.12]{BH93}).
Let $\qq'\in\Spec B$ be such that $\dim B_\qq=\dim\qq'$.
In particular, $\qq'$ is minimal and hence $\qq'\in\Ass B$.
Then $\pp':=\varphi^{-1}(\qq')$ must be minimal.
Otherwise, there is a $t\in\pp'$ not contained in any minimal prime of $A$ by prime avoidance.
Since $A$ is reduced this means that $t\in A^\reg\cap\pp'$.
Then $t$ becomes a unit in $Q(A)$ and hence $\varphi(t)\in B^\reg$ in contradiction to $\varphi(t)\in\qq'\in\Ass B$.
Since local Cohen--Macaulay rings are equidimensional (see \cite[Thm.~2.1.2.(a)]{BH93}), $\dim\pp'=\dim A$.
Applying the going up theorem to $\qq'$ and a maximal chain of primes between $\pp'$ and $\pp$ gives $\dim A=\dim B_\qq$ and the claim follows.
\end{proof}

\bibliographystyle{amsalpha}
\bibliography{blc}
\end{document}